\documentclass[12pt, reqno]{amsart}
\usepackage{amssymb,amsthm,amsmath}
\usepackage[numbers,sort&compress]{natbib}
\usepackage{amssymb,amsmath}
\usepackage{amsfonts}
\usepackage{mathrsfs}
\usepackage{latexsym}
\usepackage{amssymb}
\usepackage{amsthm}
\usepackage{indentfirst}
\let\oldsection\section
\renewcommand\section{\setcounter{equation}{0}\oldsection}

\newtheorem{theorem}{Theorem}[section]
\newtheorem{lemma}{Lemma}[section]

\newtheorem{definition}{Definition}[section]

\newtheorem{remark}{Remark}[section]

\begin{document}

\title[Global Existence of Strong Solutions to Incompressible MHD]{Global Existence of Strong Solutions to Incompressible MHD}

\author{Huajun Gong}
\address[Huajun Gong]{The Institute of Mathematical Sciences
University of Science and Technology of China, Anhui, 230026, P.R.China}
\email{huajun84@ustc.edu.cn}

\author{Jinkai~Li}
\address[Jinkai~Li]{
The Institute of Mathematical Sciences, The Chinese University of Hong Kong , Hong Kong. Also Department of Computer Science and Applied Mathematics, Weizmann Institute of Science, Rehovot 76100, Israel}
\email{jklimath@gmail.com}

\begin{abstract}
We establish the global existence and uniqueness of strong solutions to the initial boundary value problem for the incompressible MHD equations in bounded smooth domains of $\mathbb R^3$ under some suitable smallness conditions. The initial density is allowed to have vacuum, in particular, it can vanish in a set of positive Lebessgue measure. More precisely, under the assumption that the production of the quantities $\|\sqrt\rho_0u_0\|_{L^2(\Omega)}^2+\|H_0\|_{L^2(\Omega)}^2$ and $\|\nabla u_0\|_{L^2(\Omega)}^2+\|\nabla H_0\|_{L^2(\Omega)}^2$ is suitably small, with the smallness depending only on the bound of the initial density and the domain, we prove that there is a unique strong solution to the Dirichlet problem of the incompressible MHD system.
\end{abstract}

\maketitle

\section{Introduction}\label{sec1}

Magnetohydrodynamics (MHD for short) is the study of the interaction between magnetic fields and moving conducting fluids, which can be described by the following system
\begin{align}
&\rho_t+\textmd{div}(\rho u)=0,\label{1.1}\\
&\rho(u_t+(u\cdot\nabla) u)-\mu\Delta u+\nabla p=(\nabla\times H)\times H,\label{1.2}\\
&H_t-\lambda\Delta H=\nabla\times(u\times H),\label{1.3}\\
&\textmd{div}u=\textmd{div}H=0,\label{1.4}
\end{align}
where $\rho=\rho(x, t)\in\mathbb R^+$ denotes the density, $u=u(x, t)\in\mathbb R^3$ the fluid velocity, $H=H(x,t)\in\mathbb R^3$ the magnetic field, $p=p(x,t)\in\mathbb R$ the pressure, positive constant $\mu$ is called the kinematic viscosity and positive constant $\lambda$ is called the magnetic diffusivity. Usually, we refer to equation (\ref{1.1}) as the continuity equation, which represents the conservation lass of the mass, to (\ref{1.2}) as the momentum conservation equation. Equation (\ref{1.3}) is a considerably simplified version of the well known Maxiwell's equations (see \cite{Gerbeau,KULY,Landau}) by dropping Gauss's law and ignoring the displacement currents, and it's sometimes called the induction equation. As for the constraint $\textmd{div}H = 0$, it can be seen just as a restriction on the initial value $H_0$ of $H$ since $\textmd{div}H_t = 0.$ Note that, using the condition (\ref{1.4}), equations (\ref{1.2}) and (\ref{1.3}) can be rewritten as
\begin{eqnarray*}
&&\rho(u_t+(u\cdot\nabla) u)-\mu\Delta u+\nabla p=(H\cdot\nabla)H,\\
&&H_t+(u\cdot\nabla)H-\lambda\Delta H=(H\cdot\nabla)u.
\end{eqnarray*}

There are a lot of literatures on the study of MHD. Global existence of weak solutions to the homogeneous incompressible MHD were proven in \cite{Temam} and \cite{JLLions} long time ago, the density dependent case was later proven in \cite{Gerbeau} by using Lions's method \cite{Lions}. The corresponding results on the global existence of weak solutions to the compressible MHD are proven in \cite{HuWang1,HuWang2,Feireisl3,FanYu} by using the method exploited in \cite{Lions2} (see also \cite{Feireisl1,Feireisl2}), where in \cite{HuWang1} the isentropic case is considered, while in \cite{HuWang2,Feireisl3,FanYu} the non-isothermal model is considered. Local existence and uniqueness of strong solutions can be found in \cite{Tan} and \cite{WU} for the incompressible model. In \cite{Fan} and \cite{LXL1} for the compressible model, where in the last three papers the non-isothermal model are considered.
Global existence and uniqueness of strong solutions to the incompressible MHD with vacuum for two dimensional case are proved in \cite{WYH}, recently, and this work is new even for the inhomogeneous incompressible Navier-Stokes equations, however the three dimensional case is proven in \cite{LXL2} with small initial data and away from vacuum. Later the regularity criteria for the solution to the 3D MHD are proved in \cite{YZ1,YZ2,YZ3}.

One of the most important question for system (\ref{1.1})--(\ref{1.4}) is to prove the global existence and uniqueness of solutions $(\rho,u,{H})$ satisfying the initial condition
\begin{equation}
(\rho,u,H)|_{t=0}=(\rho_0,u_0,H_0)\quad\mbox{in }\Omega,\label{1.5}
\end{equation}
and boundary conditions
\begin{equation}
u|_{\partial\Omega}=0,\quad H\cdot n|_{\partial\Omega}=0,\quad(\nabla\times H)\times n|_{\partial\Omega}=0,\label{1.6}
\end{equation}
where $n$ is the unit outward norm vector on $\partial\Omega$.

The aim of this paper is to prove the global existence and uniqueness of strong solutions to system (\ref{1.1})--(\ref{1.6}) with the initial data being allowed to have vacuum. For $1\leq p\leq\infty$, we denote by $\|u\|_p$ the $L^p(\Omega)$ norm of the function $u$. The definition of strong solution is stated in the following:

\begin{definition}
We call $(\rho, u, H)$ a strong solution to the system (\ref{1.1})--(\ref{1.6}) on $(0, T)$ if $(\rho, u, H)$ satisfies (\ref{1.1})--(\ref{1.4}) a.e. in $\Omega\times(0, T)$ for some pressure function $p$, and satisfies the initial condition (\ref{1.5}) and boundary condition (\ref{1.6}), with the regularity
\begin{align*}
&\rho\in L^\infty(Q_T)\cap L^\infty(0, T; H^1(\Omega)),\quad\rho_t\in L^\infty(0, T; L^2(\Omega)),\\
&u\in L^\infty(0, T; H_0^1(\Omega)\cap H^2(\Omega))\cap L^2(0, T; W^{2,6}(\Omega)),\\
&H\in L^\infty(0, T; H^2(\Omega))\cap L^2(0, T; W^{2, 6}(\Omega)),\\
&u_t, H_t\in L^2(0, T; H^1(\Omega)),\quad\sqrt\rho u_t, H_t\in L^\infty(0, T; L^2(\Omega)).
\end{align*}
\end{definition}

Our main result is stated bellow:
\begin{theorem}\label{theorem1.1} Let $\Omega$ be a bounded smooth domain in $\mathbb{R}^3$ and $\bar\rho$ a positive number. Assume that the initial data $(\rho_0,u_0)$ satisfies the conditions
\begin{equation*}
0\leq\rho_0\leq\bar\rho, \quad\rho_0\in H^1(\Omega),\quad u_0\in H_0^1(\Omega)\cap H^2(\Omega),\quad H_0\in H^2(\Omega)
\end{equation*}
and the compatibility condition
\begin{align*}
&\textmd{div}u_0=\textmd{div}H_0=0,\quad H_0\cdot n|_{\partial\Omega}=0,\quad(\nabla\times H_0)\times n|_{\partial\Omega}=0,\\
&\Delta u_0+(H_0\cdot\nabla)H_0-\nabla p_0=\sqrt{\rho_0}g_0\quad\mbox{in }\Omega
\end{align*}
for some $(p_0, g_0)\in H^1(\Omega)\times L^2(\Omega)$.

Then there is a positive constant $\varepsilon_0$ depending only on $\bar\rho$ and $\Omega$, such that, if
\begin{equation}\label{1.8}
(\|\sqrt{\rho_0}u_0\|_2^2+\|H_0\|_2^2)(\|\nabla u_0\|_2^2+\|\nabla H_0\|_2^2)\leq\varepsilon_0,
\end{equation}
then the system (\ref{1.1})--(\ref{1.6}) has a unique global strong solution.
\end{theorem}

\begin{remark}
The quantity $(\|\sqrt{\rho}u\|_2^2+\|H\|_2^2)(\|\nabla u\|_2^2+\|\nabla H\|_2^2)$ is scaling invariable under the scaling transform
\begin{eqnarray*}
&&\rho_\lambda(x,t)=\rho(\lambda x,\lambda^2t),\quad p_\lambda(x, t)=\lambda^2p(\lambda x,\lambda^2t),\\
&&u_\lambda(x,t)=\lambda u(\lambda x,\lambda^2t),\quad H_\lambda(x,t)=\lambda H(\lambda x,\lambda^2t),
\end{eqnarray*}
and thus Theorem \ref{theorem1.1} can be viewed as a result on the global existence of strong solutions with vacuum in critical space. It seems that this is the first result in this direction, even for the Navier-Stokes equations.
\end{remark}

\section{Proof of Theorem \ref{theorem1.1}}

Throughout this section, we denote
$$
\|\sqrt{\rho_0}u_0\|_2^2+\|H_0\|_2^2=C_0^2.
$$

\begin{definition}
A finite time $T_*$ is called the finite blow-up time if
$$
\Phi(T)<\infty\quad\mbox{for all}\quad 0\leq T<T_*\quad\mbox{and}\quad\lim_{T\rightarrow T_*}\Phi(T)=\infty,
$$
where the function $\Phi(T)$ is given by
\begin{align*}
\Phi(T)=&\sup_{0\leq t\leq T}(\|(\nabla\rho, \rho_t)\|_2+\|(u, H)\|_{H^2}+\|H_t\|_2+\|\sqrt\rho u_t\|_2)\\
&+\int_0^T(\|(u, H)\|_{W^{2,6}}^2+\|(u_t, H_t)\|_{H^1}^2)dt.
\end{align*}
\end{definition}

We will use the following lemma, which states the local existence and blow up criterion of the local strong solutions.

\begin{lemma}(See \cite{Tan})\label{lem2.0}
Under the conditions of Theorem \ref{theorem1.1}  {(here we do not need the smallness condition)}, there hold

(i) (Local existence) there exists a small time $T_*$ and a unique strong solution on $(0, T_*)$,

(ii) (Blow-up criterion) $T_*$ is the finite blow-up time of $(\rho, u, H)$ if and only if
$$
\int_0^T\|(\nabla u, \nabla H)\|_2^8dt<\infty\quad\mbox{for all}\quad0<T<T_*,\quad\mbox{and}\quad\int_0^{T_*}\|(\nabla u, \nabla H)\|_2^8dt=\infty.
$$
\end{lemma}

To proof the global existence of strong solutions, we need to extend the local strong solution given in the above lemma to be a global one. For this purpose, we need do some a priori estimates on the local strong solutions. The following two lemmas give the energy estimates on the strong solutions, where the first one concerns the basic energy estimates, while the second one concerns the higher order estimates.

\begin{lemma}\label{lem2.1}
Let $(\rho, u, H)$ be a strong solution to system (\ref{1.1})--(\ref{1.6}) on $(0, T)$. Then, there holds  {
$$
(\|\sqrt\rho u\|_2^2+\|H\|_2^2)(t)+\int_0^t(\|\nabla u\|_2^2+\|\nabla H\|_2^2)ds\leq C(\|\sqrt{\rho_0}u_0\|_2^2+\|H_0\|_2^2)
$$}
for any $t\in(0, T)$.
\end{lemma}

\begin{proof}
Multiply (\ref{1.2}) by $u$ and integrate over $\Omega$, by the aid of (\ref{1.1}), we obtain after integration by parts that
\begin{align}
\frac{d}{dt}\int_\Omega\frac{\rho}{2}|u|^2dx+\int_\Omega|\nabla u|^2dx=\int_\Omega(H\cdot\nabla H)\cdot udx=-\int_\Omega(H\cdot\nabla)u\cdot Hdx.\label{2.1}
\end{align}
  {Recall the boundary condition (\ref{1.6}) and the identity
$\Delta H=\nabla\text{div}H-\nabla\times(\nabla\times H)$, it follows from integrating by parts that
\begin{align*}
-\int_\Omega\Delta H\cdot H dx=&\int_\Omega[\nabla\times(\nabla\times H)-\nabla\text{div}H]\cdot Hdx=\int_\Omega\nabla\times(\nabla\times H)\cdot Hdx\\
=&\int_{\partial\Omega}(\nabla\times H\times H)\cdot ndS+\int_\Omega|\nabla\times H|^2dx\\
=&-\int_{\partial\Omega}n\times(\nabla\times H)\cdot Hds+\int_\Omega|\nabla\times H|^2dx=\int_\Omega|\nabla\times H|^2dx.
\end{align*}}
Multiply (\ref{1.3}) by $H$ and integrate over $\Omega$,   {then it follows from the above identity that}
\begin{equation}\label{2.2}
\frac{d}{dt}\int_\Omega\frac{|H|^2}{2}dx+  {\int_\Omega|\nabla\times  H|^2dx}=\int_\Omega(H\cdot\nabla)u\cdot Hdx.
\end{equation}
Summing (\ref{2.1}) with (\ref{2.2}) up, and integrating the resulting equation with respect to $t$, we obtain
\begin{equation*}
\int_\Omega(\rho|u|^2+|H|^2)dx+2\int_0^t\int_\Omega(|\nabla u|^2+  {|\nabla\times H|^2)dxds}=\int_\Omega(\rho_0|u_0|^2+|H_0|^2)dx,
\end{equation*}
  {
and then by using \cite{WAHL}, we have
\begin{equation*}(\|\sqrt\rho u\|_2^2+\|H\|_2^2)(t)+\int_0^t(\|\nabla u\|_2^2+\|\nabla H\|_2^2)ds\leq C(\|\sqrt{\rho_0}u_0\|_2^2+\|H_0\|_2^2),
\end{equation*}}
completing the proof.
\end{proof}

\begin{lemma}\label{lem2.2}
Let $(\rho, u, H)$ be a strong solution to system (\ref{1.1})--(\ref{1.6}) on $(0, T)$. Then, there holds
\begin{align*}
&\sup_{0\leq s\leq t}(\|\nabla u\|_2^2+\|\nabla H\|_2^2)+\int_0^t(\|\nabla^2u\|_2^2+\|\nabla^2 H\|_2^2)ds\\
\leq&2(\|\nabla u_0\|_2^2+\|\nabla H_0\|_2^2)+C\sup_{0\leq s\leq t}(C_0^2\|\nabla u\|_2^4+C_0\|\nabla H\|_2^3)\\
&+CC_0\sup_{0\leq s\leq t}(\|\nabla u\|_2+\|\nabla H\|_2)\int_0^t(\|\nabla^2u\|_2^2+\|\nabla^2H\|_2^2)ds
\end{align*}
for any $t\in(0, T)$.
\end{lemma}

\begin{proof}
Multiplying (\ref{1.1}) by $u_t$ and integration by parts yields
\allowdisplaybreaks\begin{align*}
&\frac{d}{dt}\int_\Omega\frac{|\nabla u|^2}{2}dx+\int_\Omega\rho|u_t|^2dx\\
=&\int_\Omega[(H\cdot\nabla)H\cdot u_t-\rho(u\cdot\nabla)u\cdot u_t]dx=-\int_\Omega[\rho(u\cdot\nabla)u\cdot u_t+(H\cdot\nabla)u_t\cdot H]dx\\
=&-\frac{d}{dt}\int_\Omega (H\cdot\nabla)u\cdot Hdx+\int_\Omega[(H_t\cdot\nabla)u\cdot H+(H\cdot\nabla)u\cdot H_t-\rho(u\cdot\nabla)u\cdot u_t]dx\\
\leq&-\frac{d}{dt}\int_\Omega (H\cdot\nabla)u\cdot Hdx+\int_\Omega\left(\frac{1}{2}\rho|u_t|^2+\frac{1}{4}|H_t|^2\right)dx+4\int_\Omega(|H|^2|\nabla u|^2+\rho|u|^2|\nabla u|^2)dx,
\end{align*}
and thus
\begin{align}
&\frac{d}{dt}\int_\Omega\left(\frac{|\nabla u|^2}{2}+(H\cdot\nabla)u\cdot H\right)dx+\int_\Omega\rho|u_t|^2dx\nonumber\\
\leq&\int_\Omega\left(\frac{1}{2}\rho|u_t|^2+\frac{1}{4}|H_t|^2\right)dx+4\int_\Omega(|H|^2|\nabla u|^2+\rho|u|^2|\nabla u|^2)dx. \label{2.3}
\end{align}
Multiplying (\ref{1.3}) by $H_t$ and integration by parts yields
\begin{align}
&\frac{d}{dt}\int_\Omega\frac{|\nabla H|^2}{2}dx+\int_\Omega|H_t|^2dx=\int_\Omega[(H\cdot\nabla)u-(u\cdot\nabla)H]H_tdx\nonumber\\
\leq&\frac{1}{4}\int_\Omega|H_t|^2dx+4\int_\Omega(|H|^2|\nabla u|^2+|u|^2|\nabla H|^2)dx.\label{2.4}
\end{align}
Summing (\ref{2.3}) with (\ref{2.4}) up, it follows
\begin{align*}
&\frac{d}{dt}\int_\Omega(|\nabla u|^2+|\nabla H|^2+2(H\cdot\nabla)u\cdot H)dx+\int_\Omega(\rho|u_t|^2+|H_t|^2)dx\\
\leq&  {16}\int_\Omega(|H|^2|\nabla u|^2+|u|^2|\nabla H|^2+\rho|u|^2|\nabla u|^2)dx,
\end{align*}
and thus
\begin{align}
&\sup_{0\leq s\leq t}(\|\nabla u\|_2^2+\|\nabla H\|_2^2)+\int_0^t(\|\sqrt\rho u_t\|_2^2+\|H_t\|_2^2)ds\nonumber\\
\leq&(\|\nabla u_0\|_2^2+\|\nabla H_0\|_2^2)+4\sup_{0\leq s\leq t}\int_\Omega|H|^2|\nabla u|dx\nonumber\\
&+10\int_0^t\int_\Omega(|H|^2|\nabla u|^2+|u|^2|\nabla H|^2+\rho|u|^2|\nabla u|^2)dxds.\label{2.5}
\end{align}
Applying $H^2$ estimates to Stokes equations and elliptic equations, it follows from (\ref{1.2}) and (\ref{1.3}) that
\begin{align}
\|\nabla^2u\|_2^2\leq& C(\|\rho u_t\|_2^2+\|\rho(u\cdot\nabla)u\|_2^2+\|(H\cdot\nabla)H\|_2^2)\nonumber\\
\leq&C(\|\sqrt\rho u_t\|_2^2+\|\sqrt\rho(u\cdot\nabla)u\|_2^2+\|(H\cdot\nabla)H\|_2^2)\label{2.6}
\end{align}
and
\begin{equation}
\|\nabla^2H\|_2^2\leq C(\|H_t\|_2^2+\|(u\cdot\nabla)H\|_2^2+\|(H\cdot\nabla)u\|_2^2),\label{2.7}
\end{equation}
where we have used the fact that $0\leq\rho\leq\bar\rho$,   {which follows from equation (\ref{1.1}) by the characteristic method.} Combining (\ref{2.5}) with (\ref{2.6}), together with (\ref{2.7}), we obtain
\begin{align}
&\sup_{0\leq s\leq t}(\|\nabla u\|_2^2+\|\nabla H\|_2^2)+\int_0^t(\|\sqrt\rho u_t\|_2^2+\|\nabla^2u\|_2^2+\|H_t\|_2^2+\|\nabla^2H\|_2^2)ds\nonumber\\
\leq&(\|\nabla u_0\|_2^2+\|\nabla H_0\|_2^2)+4\sup_{0\leq s\leq t}\int_\Omega|H|^2|\nabla u|dx\nonumber\\
&+C\int_0^t\int_\Omega[|H|^2(|\nabla u|^2+|\nabla H|^2)+|u|^2|\nabla H|^2+\rho|u|^2|\nabla u|^2]dxds.\label{2.8}
\end{align}
  {We have use the facts that}
\begin{align*}
&  {\|H\|_6\leq C\|\nabla H\|_2, \quad\|u\|_6\leq C\|\nabla u\|_2},\\
&  {\|\nabla u\|_6\leq C\|u\|_{H^2}\leq C\|\Delta u\|_2,\quad\|\nabla H\|_6\leq C\|H\|_{H^2}\leq C\|\Delta H\|_2,}
\end{align*}
  {where, the first line can be easily proved by contradiction arguments based on the well known compact embedding result $H^1\hookrightarrow\hookrightarrow L^2(\Omega)$, while the second line follows from the elliptic estimates}.
Using these inequalities, it follows from the H\"older inequality, Sobolev inequality and Young inequality that
\begin{align*}
\int_\Omega|H|^2|\nabla u|dx\leq&\|H\|_4^2\|\nabla u\|_2\leq\|H\|_2^{1/2}\|H\|_6^{3/2}\|\nabla u\|_2\\
\leq&C\|H\|_2^{1/2}\|\nabla H\|_2^{3/2}\|\nabla u\|_2\leq\varepsilon\|\nabla u\|_2^2+C\|H\|_2\|\nabla H\|_2^3,\\
\int_\Omega\rho|u|^2|\nabla u|^2dx\leq&C\|\sqrt\rho u\|_2\|u\|_6\|\nabla u\|_6^2\leq C\|\sqrt\rho u\|_2\|\nabla u \|_2\|\nabla^2u\|_2^2,\\
\int_\Omega|u|^2|\nabla H|^2dx\leq&\|u\|_6^2\|\nabla H\|_3^2\leq C\|\nabla u\|_2^2\|\nabla H\|_2\|\nabla H\|_6\\
\leq &C\|\nabla u\|_2^2\|\nabla H\|_2\|\nabla^2H\|_2\leq\varepsilon\|\nabla^2H\|_2^2+C\|\nabla u\|_2^4\|\nabla H\|_2^2,\\
\int_\Omega|H|^2(|\nabla u|^2+|\nabla H|^2)d  {x}\leq&\|H\|_2\|H\|_6(\|\nabla u\|_6^2+\|\nabla H\|_6^2)\\
\leq& C\|H\|_2\|\nabla H\|_2(\|\nabla^2u\|_2^2+\|\nabla^2H\|_2^2).
\end{align*}
Substituting the above inequalities into (\ref{2.8}), taking $\varepsilon$ small enough, it follows from Lemma \ref{lem2.1} that
\begin{align*}
&\sup_{0\leq s\leq t}(\|\nabla u\|_2^2+\|\nabla H\|_2^2)+\int_0^t(\|\sqrt\rho u_t\|_2^2+\|\nabla^2u\|_2^2+\|H_t\|_2^2+\|\nabla H\|_2^2)ds\\
\leq&2(\|\nabla u_0\|_2^2+\|\nabla H_0\|_2^2)+C\sup_{0\leq s\leq t}\|H\|_2\|\nabla H\|_2^3+C\sup_{0\leq s\leq t}(\|\sqrt\rho u\|_2\|\nabla u\|_2+\|H\|_2\|\nabla H\|_2)\\
&\times\int_0^t(\|\nabla^2u\|_2^2+\|\nabla^2H\|_2^2)ds+C\sup_{0\leq s\leq t}\|\nabla u\|_2^4\int_0^t\|\nabla H\|_2^2ds\\
\leq&2(\|\nabla u_0\|_2^2+\|\nabla H_0\|_2^2)+C\sup_{0\leq s\leq t}(C_0^2\|\nabla u\|_2^4+C_0\|\nabla H\|_2^3)\\
&+CC_0\sup_{0\leq s\leq t}(\|\nabla u\|_2+\|\nabla H\|_2)\int_0^t(\|\nabla^2u\|_2^2+\|\nabla^2H\|_2^2)ds,
\end{align*}
completing the proof.
\end{proof}

By using the two lemmas above, we can prove the following a priori estimates.

\begin{lemma}\label{lem2.3}
Let $(\rho, u, H)$ be a strong solution to system (\ref{1.1})--(\ref{1.6}) on $(0, T)$. Then, there exists a positive constant $\varepsilon_0$ depending only on $\bar\rho$ and $\Omega$, such that
$$
\sup_{0\leq t\leq T}(\|\nabla u\|_2^2+\|\nabla H\|_2^2)+\int_0^T(\|\nabla^2u\|_2^2+\|\nabla H\|_2^2)dt
\leq 4(\|\nabla u_0\|_2^2+\|\nabla H_0\|_2^2),
$$
provided
$$
C_0(\|\nabla u_0\|_2^2+\|\nabla H_0\|_2^2)\leq\varepsilon_0.
$$
\end{lemma}

\begin{proof}
Define functions $E(t)$ and $\Phi(t)$ as follows
\begin{eqnarray*}
&E(t)=\displaystyle\sup_{0\leq s\leq t}(\|\nabla u\|_2^2+\|\nabla H\|_2^2)+\int_0^t(\|\nabla^2u\|_2^2+\|\nabla^2H\|_2^2)ds,\\
&\Phi(t)=C_0^2\displaystyle\sup_{0\leq s\leq t}(\|\nabla u\|_2^2+\|\nabla H\|_2^2).
\end{eqnarray*}
In view of the regularities of $u$ and $H$, one can easily check that both $E(t)$ and $\Phi(t)$ are continuous functions on $[0, T]$. By Lemma \ref{lem2.2}, there is a positive constant $C_*$, such that
\begin{equation}\label{2.9}
E(t)\leq 2(\|\nabla u_0\|_2^2+\|\nabla H_0\|_2^2)+C_*[\Phi(t)^{1/2}+\Phi(t)]E(t).
\end{equation}

We take
$$
\varepsilon_0=\min\left\{\frac{1}{32C_*}, \frac{1}{128C_*^2}\right\},
$$
and suppose that
$$
C_0^2(\|\nabla u_0\|_2^2+\|\nabla H_0\|_2^2)\leq\varepsilon_0.
$$

We claim that
$$
\Phi(t)<\min\left\{\frac{1}{4C_*}, \frac{1}{16C_*^2}\right\}, \qquad0\leq t\leq T.
$$
Otherwise, by the continuity and monotonicity of $\Phi(t)$, there is $T_0\in(0, T]$, such that
\begin{equation}\label{2.10}
\Phi(T_0)=\min\left\{\frac{1}{4C_*}, \frac{1}{16C_*^2}\right\}.
\end{equation}
On account of (\ref{2.10}), it follows from (\ref{2.9}) that
$$
E(T_0)\leq 2(\|\nabla u_0\|_2^2+\|\nabla H_0\|_2^2)+\frac{1}{2}E(T_0),
$$
and thus
$$
E(T_0)\leq 4(\|\nabla u_0\|_2^2+\|\nabla H_0\|_2^2).
$$
Recalling the definition of $E(t)$ and $\Phi(t)$, we deduce from the above inequality that
\begin{align*}
\Phi(T_0)\leq&C_0^2E(T_0)\leq 4C_0^2(\|\nabla u_0\|_2^2+\|\nabla H_0\|_2^2)\leq4\varepsilon_0=\min\left\{\frac{1}{8C_*}, \frac{1}{32C_*^2}\right\},
\end{align*}
contradicting to (\ref{2.10}). This contradiction implies that the claim is true.

By the aid of the claim we proved in the above, it follows from (\ref{2.9}) that
$$
E(t)\leq 4(\|\nabla u_0\|_2^2+\|\nabla H_0\|_2^2),\qquad 0< t< T,
$$
completing the proof.
\end{proof}

Now, we can give the proof of Theorem \ref{theorem1.1}.

\textbf{Proof of Theorem \ref{theorem1.1}: }Let $\varepsilon_0$ be the constant stated in Lemma \ref{lem2.3} and suppose that the initial data satisfies
$$
(\|\sqrt{\rho_0}u_0\|_2^2+\|H_0\|_2^2)(\|\nabla u_0\|_2^2+\|\nabla H_0\|_2^2)\leq\varepsilon_0.
$$
By Lemma \ref{lem2.0}, there is a unique strong solution $(\rho, u, H)$ to system (\ref{1.1})--(\ref{1.6}). Extend such local solution to the maximal existence time interval $[0, T_*)$. We will prove that $T_*=\infty$. Suppose, by contradiction, that $T_*<\infty$. By Lemma \ref{lem2.0}, the time $T_*$ can be characterized as follows
$$
\int_0^T\|(\nabla u, \nabla H)\|_2^8dt<\infty,\quad\mbox{for all}\quad 0<T<T_*,
$$
and
\begin{equation}\label{2.12}
\int_0^{T_*}\|(\nabla u, \nabla H)\|_2^8dt=\infty.
\end{equation}
By Lemma \ref{lem2.3}, for any $0<T<T_*$, there holds
$$
\sup_{0\leq t\leq T}(\|\nabla u\|_2^2+\|\nabla H\|_2^2)\leq 4(\|\nabla u_0\|_2^2+\|\nabla H_0\|_2^2),
$$
and thus
$$
\sup_{0\leq t< T_*}(\|\nabla u\|_2^2+\|\nabla H\|_2^2)\leq 4(\|\nabla u_0\|_2^2+\|\nabla H_0\|_2^2),
$$
which implies
$$
\int_0^{T_*}\|(\nabla u, \nabla H)\|_2^8dt\leq   {4^{4}(\|\nabla u_0\|_2^2+\|\nabla H_0\|_2^2)^4T_*}<\infty,
$$
contradicting to (\ref{2.12}). This contradiction provides us that $T_*=\infty$, and thus we obtain a global strong solution. The proof is complete.\\

  {\textbf{Acknowledgement.} The first author is partially supported by CUSF WK0010000028. The authors would also like to thank the referee for his/her comments and suggestions.}

\end{document}